\documentclass{article}
\usepackage[titletoc,toc]{appendix}
\usepackage{amsmath,amssymb,amsthm,color,graphicx,setspace,verbatim}
\usepackage{hyperref}

\usepackage{algorithmic}

\newcommand{\dgeq}{\succeq}
\newcommand{\fm}{{\mathfrak m}}
\newcommand{\gr}{{\mathfrak{gr}}}
\newcommand{\p}{{\partial}}

\DeclareMathOperator{\ord}{ord}

\DeclareMathOperator{\Dual}{Dual}
\DeclareMathOperator{\POStoN}{(\bZ_{\geq 0})^N}

\DeclareMathOperator{\HP}{HP} 

\DeclareMathOperator{\initial}{in}

\newtheorem{theorem}{Theorem}[section]
\newtheorem{lemma}[theorem]{Lemma}

\newtheorem{corollary}[theorem]{Corollary}
\newtheorem{definition}[theorem]{Definition}
\newtheorem{example}[theorem]{Example}
\newtheorem{proposition}[theorem]{Proposition}

\newtheorem{algorithm}[theorem]{Algorithm}
\newtheorem{remark}[theorem]{Remark}

\newcommand{\bN}{{\mathbb N}}

\newcommand{\bZ}{{\mathbb Z}}

\newcommand{\bC}{{\mathbb C}}

\newcommand{\bV}{{\mathbb V}}

\newcommand{\Span}{\operatorname{span}}

\newcommand\Var{{\bV}}
\def\gcorner{g-corner}
\DeclareMathOperator{\der}{\sigma}
\DeclareMathOperator{\lcm}{lcm}
\newcommand{\ideal}[1]{\langle #1 \rangle}
\newcommand{\LT}{\operatorname{in}_\geq}
\newcommand{\LTg}{\operatorname{in}_\succeq}

  \def \tab#1{\kern #1 truein}

\begin{document}
\title{Eliminating dual spaces}
\author{
Robert Krone\thanks{School of Mathematics, Georgia Tech, Atlanta GA, USA ({\tt rkrone3@math.gatech.edu}). Partially supported by NSF grant DMS-1151297}
\and
Anton Leykin\thanks{School of Mathematics, Georgia Tech, Atlanta GA, USA ({\tt leykin@math.gatech.edu}). Partially supported by NSF grants DMS-0914802 and DMS-1151297}}
\maketitle

\begin{abstract}
Macaulay dual spaces provide a local description of an affine scheme and give rise to computational machinery that is compatible with the methods of numerical algebraic geometry. We introduce eliminating dual spaces, use them for computing dual spaces of quotient ideals, and develop an algorithm for detection of embedded points on an algebraic curve.  
\end{abstract}


\section{Introduction}

We explore algorithms that use dual spaces to compute local information about polynomial ideals over the complex numbers, as an alternative to working directly with polynomials.  This strategy often has advantages in the regime of approximate numerical computations and for hybrid symbolic-numerical algorithms.  Many computations are reduced to linear algebra, allowing numerical linear algebra techniques to be applied.

Given generators of a polynomial ideal $I$ and a point $p$ in its vanishing set, the dual space of $I$ at $p$ is the vector space dual of the extension of $I$ in the local ring at $p$, and it uniquely encodes the local properties of $I$ there.  Certain combinatorial information about the dual space, such as dimension, can be accurately computed even when $p$ is only known approximately but with high enough precision. Existing methods of numerical algebraic geometry (see e.g., \cite{SVW9} and~\cite{Sommese-Wampler-book-05}) offer efficient algorithms to approximately compute points on the vanishing set of an ideal.  

The contributions of this paper are:
\begin{itemize}
 \item Establishing a correspondence between a local monomial order on a local ring (primal order) and the corresponding dual order on the monomials of the dual space.
 \item Characterizing the dual space of colon (quotient) ideals.
 \item Introducing the notion of {\em eliminating dual spaces} as a computational tool.
 \item Applying these ideas to construct an algorithm for detecting embedded primary components of a curve.
\end{itemize}
Additionally, we consolidate the necessary background on dual spaces, state their properties in a modern language, and provide proofs where references are unavailable.

The idea of studying systems of polynomials through dual spaces dates back to Macaulay~\cite{M16}.  
Most of the recent work using Macaulay's machinery concerns zero-dimensional ideals or, geometrically speaking, isolated points. This includes algorithms for computing a basis of the dual space~\cite{Mourrain:inverse-systems,DZ05} and the local Hilbert function at an isolated point~\cite{griffin2011numerical}, as well as various deflation procedures~\cite{Lec-deflation-02,LVZ,Hauenstein-Wampler:isosingular}. 
Several studies depart from the zero-dimensional setting: the local dimension test~\cite{Bates-et-al:local-dimension-test}, computations using dual spaces for homogeneous ideals~\cite{Hauenstein:preprint2011}, and an numerical algorithm for local Hilbert polynomial in the general case~\cite{krone2014numerical}.

As an application of eliminating dual spaces we present an algorithm for detecting embedded components of a curve. This fills in a key missing piece in the program to numerically compute a primary decomposition of a polynomial ideal as laid out in \cite{Leykin:NPD} in the case of one-dimensional ideals.  We remark that eliminating dual spaces may, in theory, be used to provide an embedded component test in dimension more than one. However, we did not pursue this direction since for practical computation this potential technique appears to be inferior to an alternative treatment we develop in~\cite{krone2014numerical}. Our algorithms for dual spaces and detecting embedded components are implemented in the {\em Macaulay2} computer algebra system~\cite{M2www} and the code is posted at~\cite{ECTwww}.

The rest of the paper is structured as follows.  Section~\ref{Sec:dualSpace} gives an expository background on dual spaces of ideals in a local ring, laying out the facts that will be used later in the paper and giving proofs where they were not readily available in the literature.  Additionally we characterize the dual space of quotient ideals.  Section~\ref{Sec:Hilbert} describes the correspondence between the initial terms of an ideal with respect to a (primal) local monomial order and the initial terms of the dual space in the dual order.  Such a relation has been observed before in special settings such as zero-dimensional ideals with a graded order, but we give a fully general result.  This section also describes the relation between truncated dual spaces and the Hilbert function.  Section~\ref{Sec:eliminating} introduces the eliminating dual space, which generalize the truncated dual space, and can be used to compute truncated dual spaces of a quotient ideal.  These ideas are applied in~\ref{Sec:embedded-test} to the problem of detecting the embedded components on a curve, and examples are provided.

\section{Macaulay dual spaces}\label{Sec:dualSpace}

For $\alpha \in \POStoN$ and $y\in\bC^N$, define
\begin{itemize}
\item $x^\alpha = x_1^{\alpha_1}\cdots x_N^{\alpha_N}$,
\item $|\alpha| = \sum_{i=1}^N \alpha_i$,
\item $\alpha! = \alpha_1!\alpha_2!\dots\alpha_N!$,
\item $\p^\alpha = \frac{1}{\alpha!}\frac{\p^{|\alpha|}}{\p x^\alpha}$, and
\item $\p^\alpha[y]: R \rightarrow \bC$~~defined by~
    $\p^\alpha[y](g) = (\p^\alpha g)(y)$.
\end{itemize}
The differential functional $\partial^\alpha[y]$ sometimes would be written $\partial^{x^\alpha}[y]$ (e.g. $\partial^1 - \partial^{y} + \partial^{x^2yz}$) and
when the point $y$ is implied $\partial^\alpha[y]$ would be written as~$\partial^\alpha$.
For $y \in \bC^N$, let $D_y = \Span_{\bC}\left\{\partial^\alpha[y]~|~\alpha \in \POStoN \right\}$
be the vector space of differential functionals at $y$. This linear space is graded by {\em order}, for a finite sum $q = \sum c_\alpha\p^\alpha$,
\[
\ord q = \max_{c_\alpha\neq 0} |\alpha|.
\]
The {\em homogeneous} part of order $i$ of $q\in D_y$ is referred to as $q_i$. This grading is the associated graded linear space of the filtration $D_y^*$:
\[D_y^0 \subset D_y^1 \subset D_y^2 \subset \ldots \text{, where }D_y^i = \{q\in D_y~|~\ord q\leq i\}\}.\]

\begin{definition}
The {\it Macaulay dual space}, or simply {\em dual space},
of differential functionals that vanish
at $y$ for an ideal $I\subset\bC[x]=\bC[x_1,\dots,x_N]$~is
\begin{equation}\label{Eq:Dual}
D_y[I] = \{q\in D_y~|~q(g)=0\hbox{~for all~}g\in I\}.
\end{equation}
The dual space $D_y[I]$ is a linear subspace of $D_y$, a basis of $D_y[I]$ is called a {\it dual basis} for $I$.
\end{definition}

The following theorem of Macaulay describes the dimension of the dual space at an isolated solution $y$. The following statement appears in the classical text of Macaulay~\cite{M16}.
\begin{theorem}\label{Thm:dimDual}
A solution $y\in\Var(I)$ is isolated with multiplicity $m$ if and only if $\dim_\bC D_y[I] = m$.
\end{theorem}

\begin{definition} A subspace $S \subset D_y$ is {\em homogeneous} if it is spanned by homogeneous functionals $q\in D_y^{\ord q}\setminus D_y^{\ord q-1}$. If, in addition, $S$ is spanned by $\p^\alpha[y],\ \alpha\in A$ for some subset $A \subset (\bZ_{\geq 0})^N$, then $S$ is called {\em monomial}.
\end{definition}

\subsection{Local ring vs. Dual space}\label{Sec:Local-vs-Dual}

For the purpose of this section, without a loss of generality, we may assume $y=0\in\bC^N$.
Consider the local ring $R_0 = R_\fm$ where $\fm = (x_1,\ldots,x_N)$. Let the space of dual functionals be defined as above replacing $R$ (polynomial) with $R_0$ (rational functions with denominators not vanishing at 0).

\begin{remark}\label{Rem:OneToOne} Ideals in $R$ with no primary components away
from the origin are in one-to-one correspondence with ideals in the local ring $R_0$:
\begin{itemize}
\item an ideal $I\subset R$ defines the extension $IR_0\subset R_0$;
\item an ideal $I\subset R_0$ corresponds to the ideal $I\cap R \subset R$
with no primary components away from the origin.
\end{itemize}
\end{remark}

\begin{proposition}
For ideal $I \subset R$, the dual space $D_0[I]$ is identical to the dual space of its extension in $R_0$, $D_0[IR_0]$. 
\end{proposition}
\begin{proof}
Any rational function $g \in R_0$ can be expressed as a power series $g = \sum_\alpha c_\alpha x^\alpha$.  If $q \in D_0[I]$, then for any $f \in I$, $q(x^\alpha f) = 0$ for all monomials $x^\alpha$.  Then
\[ q(gf) = \sum_{\alpha \in (\bZ_{\geq 0})^N} c_\alpha q(x^\alpha f) = 0 \]
so $q \in D_0[IR_0]$.  For $q$ not in $D_0[I]$ there is some $f\in I$ with $q(f) \neq 0$, and $f$ is also in $IR_0$.
\end{proof}
As a result we will speak interchangeably about the dual space of an ideal $I$ at the point $0$ and the dual space of its extension in the localization of $R$ at $0$, $IR_0$.

The following lemma provides another characterization of the extension of an ideal $I$ in the local ring, which will help describe the close connection between $IR_0$ and the Macaulay dual space.

\begin{lemma}\label{Lem:ExtensionAsIntersection}
 For any ideal $I \subset R$,
 \[ IR_0 \cap R = \bigcap_{k=1}^\infty (I + \fm^k). \]
\end{lemma}
\begin{proof}
 Let $\hat{R}$ denote the completion of $R$ with respect to the maximal ideal $\fm$ (the formal power series ring $\hat{R} = \bC [[x_1,\ldots,x_N]]$).  The kernel of the map of $R$-modules $R/I \to \widehat{R/I}$ is $\bigcap_k \fm^k(R/I)$, and by the exactness of completion $\widehat{R/I} \cong \hat{R}/I\hat{R}$ (see \cite{Atiyah-Macdonald} Chapter 10).  Composing the quotient map $R \to R/I$ with the above, we see that $I\hat{R} \cap R$, which is the kernel of natural map $R \to \hat{R}/I\hat{R}$, is $\bigcap_k (I + \fm^k)$.
 
 For any $f \in I\hat{R} \cap R$, there is $h \in I$, $g \in \hat{R}$ such that $f = hg$, so $g = h/f$ is a rational function in $R_0$.  Therefore $f \in IR_0 \cap R$, and so $IR_0 \cap R = I\hat{R} \cap R$.
\end{proof}

\begin{proposition}\label{Prop:IdealFromDual}
For ideal $I \subset R$, $f \in IR_0 \cap R$ if and only if $q(f) = 0$ for all $q \in D_0[I]$.
\end{proposition}
\begin{proof}
It follows from the definition that $f \in I$ implies $q(f) = 0$ for all $q \in D_0[I]$.

 Let $R^k$ be the space of polynomials with degree $\leq k$, let $f^k$ denote the truncation of $f$ to degree $k$ and let $I^k \subset R^k$ be the set $\{ f^k : f \in I\}$.  Since $R^k$ is a finite dimensional vector space,
 \[ (I^k)^\perp = D_0^k[I] = D_0[I + \fm^{k+1}]. \]
 Suppose for some polynomial $f$ that $q(f) = 0$ for all $q \in D_0[I] = \bigcup_k D_0^k[I]$.  Because $(I^k)^{\perp\perp} = I^k$, we have $f^k \in I^k$ which implies $f \in I + \fm^{k+1}$.  By Lemma~\ref{Lem:ExtensionAsIntersection}, $f \in IR_0 \cap R$.
\end{proof}

\begin{corollary}\label{Coro:DualIsInjective}
For ideals $J_1,J_2 \subset R_0$, $J_1 \subset J_2$ if and only if $D_0[J_1] \supset D_0[J_2]$.
\end{corollary}
\begin{proof}
It's clear that $J_1 \subset J_2$ implies $D_0[J_1] \supset D_0[J_2]$.  Suppose $J_1 \not\subset J_2$, so there is polynomial $f \in J_1$ with $f \notin J_2$.  By Proposition~\ref{Prop:IdealFromDual} there is $q \in D_0[J_2]$ with $q(f) \neq 0$, so $D_0[J_2] \not\subset D_0[J_1]$.
\end{proof}
An immediate consequence of this corollary is that an ideal $J \subset R_0$ is uniquely determined by its dual space $D_0[J]$.

\begin{corollary}\label{Coro:HomoMono}
The dual space $D_0[J]$ is homogeneous (respectively, monomial) iff the ideal $J\subset R_0$ is homogeneous (respectively, monomial), i.e., generated by homogeneous elements with respect to filtration
$\{\fm^k\}_{k\geq 0}$ (respectively, by monomials).
\end{corollary}
\begin{proof}
Given a homogeneous (respectively, monomial) dual space $L=D_0[J]$ of an ideal $J \subset R_0$ it is straightforward to write down homogeneous (respectively, monomial) $I\subset R$ such that $D_0[I] = L$. Namely, its homogeneous part of order $k$ is the set of polynomials orthogonal to $L^k/L^{k-1}$; for the monomial case, it is particularly explicit: a monomial $x^\alpha$ belongs to $I$ iff $\p^\alpha\notin L$. The extension $IR_0$ is determined by $L$ uniquely according to Proposition~\ref{Coro:DualIsInjective}, hence, $IR_0 = J$.
\end{proof}

\begin{remark}
  One could easily extend the definition of homogeneous and monomial ideals to the local ring $R_y$ for an arbitrary point $y\in\bC^N$: in particular, an ideal is called monomial if it is generated by elements of the form $(x-y)^\alpha$, $\alpha\in \POStoN$.
\end{remark}

Macaulay dual bases allow for testing ideal membership at a solution \cite{MMM96}
as stated in the following proposition. This can be readily generalized for homogeneous ideals using the following corollary, Remark~\ref{Rem:OneToOne}, and Proposition~\ref{Coro:DualIsInjective}.

\begin{corollary}[Lemma~11 of \cite{Hauenstein:preprint2011}]\label{Coro:HomoMembership}
A polynomial $f \in R$ of degree $d$ is a member of a homogeneous ideal $I\subset R$ iff $f$ is annihilated by $D_0^{d}[I]$.
\end{corollary}
\begin{proof}
It follows from the proof of Corollary~\ref{Coro:HomoMono} that $D_0^d[I]$ determines $J = I R/\fm^{d+1}$.
Now, $f\in I$ iff its image $\bar f \in J$ iff $f$ is annihilated by $D_0^{d}[I]$.
\end{proof}

The statement of Corollary~\ref{Coro:HomoMembership} corrects that of Theorem~$4.6$ of~\cite{Leykin:NPD} where the assumption of homogeneity was missed as shown in \cite{DualExample}. The local membership test without the assumption of homogeneity is a much harder task, addressed in \cite{krone2014numerical}.

\subsection{Action of differentiation on the dual space}
An alternative characterization of the dual space can be given via Proposition~\ref{Prop:closedness}. There is a natural action of $R_0$ on $D_0$ by pre-multiplication.  Specifically for $q \in D_0$ and $g \in R_0$ let $g\cdot q \in D_0$ denote the functional defined by $(g\cdot q)(f) = q(gf)$.  It can be checked that this gives $D_0$ an $R_0$-module structure. The action of each variable $x_i$ can also be considered as {\em differentiating} functionals in $D_0$ by $\p_i$ (up to normalization).  Let $\der_{x_i}: D_0 \to D_0$ denote the map defined by the action of $x_i$.
\begin{eqnarray*}
  \der_{x_i}:  D_0 &\to& D_0\\
        \p^\alpha &\mapsto& \p^{\alpha - e_i}, \ \ \ \ (i=1,\ldots,N),
\end{eqnarray*}
where $\p^\beta$ is taken to be $0$ when any entry of $\beta$ is less than zero.

The following statements (from Proposition~\ref{Prop:closed} to Corollary~\ref{Coro:truncated-for-homogeneous}) appear, perhaps in alternative phrasing, in many works addressing the duality at hand (see, for example,~\cite{Mourrain:inverse-systems}). We collect the essential pieces, stated in our language, and complete with our own short proofs to guide reader's intuition for this paper. 

\begin{proposition}\label{Prop:closed}
For a subspace $L \subset D_0$ the following are equivalent:
\begin{itemize}
 \item $L$ is the dual space of some ideal \mbox{$J_L \subset R_0$}.
 \item $L$ is closed under differentiation by each variable: $x_i\cdot L \subset L$ for all \mbox{$1 \leq i \leq N$}.
 \item $L$ is an $R_0$-submodule of $D_0$.
\end{itemize}
\end{proposition}
\begin{proof}
For any $L \subset D_0$ define

\[ J_L = \{ f\in R_0: q(f) = 0 \text{ for all } q \in L\}. \]

If $L$ is closed under differentiation, then $J_L$ is closed under multiplication by each $x_i$, and therefore under multiplication by all monomials in $R_0$.  Express any $g \in R_0$ as $g = \sum_\alpha c_\alpha x^\alpha$.  Then if $f \in J_L$ and $q \in L$, $q(gf) = \sum_\alpha c_\alpha q(x^\alpha f)$ and each term is zero, so $gf \in J_L$.  Therefore $J_L$ is an ideal and $D_0[J_L] = L$.  Conversely if $L = D_0[J_L]$ and $q \in L$ then $(x_i\cdot q)(f) = q(x_i f) = 0$ for all $f \in J_L$ so $x_i\cdot q \in L$.
\end{proof}

Consider the map $$\Dual : \{\text{ideals of } R_0\}  \to \{R_0\text{-submodules of }D_0\}$$ defined by $\Dual(J) = D_0[J]$.  By Corollary~\ref{Coro:DualIsInjective} and Proposition~\ref{Prop:closed}, this map is a bijection.  This provides another way to characterize the dual space.

\begin{proposition}\label{Prop:closedness}
For ideal $J=\langle f_1,\dots,f_n\rangle\subset R_0$, let $L$ be the maximal $R_0$-submodule of $D_0$ that satisfies $q(f_i) = 0$ for all $q \in L$ and all $0 \leq i \leq n$.  Then $L = D_0[J]$.
\end{proposition}
\begin{proof}
$D_0[J]$ is closed under differentiation and satisfies $q(f_i) = 0$ for all $q \in D_0[J]$ and $0 \leq i \leq n$, so $D_0[J] \subseteq L$.  The ideal $J_L$ contains $\{f_1,\ldots,f_n\}$, so $J \subseteq J_L$ which implies $L \subseteq D_0[J]$.
\end{proof}

\begin{remark}
 For an ideal $J \subset R_0$, the dual space $D_0[J]$ is finitely-generated as an $R_0$-module only when it is a finite dimensional vector space.  If $D_0[J]$ is generated by a single functional $p$, then $J$ is exactly the {\em apolar ideal} of $p$~(see, for instance, \cite{iarrobino1999power} for the definition). 
\end{remark}

A result of Proposition~\ref{Prop:closedness} is that for $I = \ideal{f_1,\dots,f_n}$, a dual element $q$ is in $D_0[I]$ if and only if $q(f_i) = 0$ and $x_j\cdot q \in D_0[J]$ for each $0 \leq i \leq n$ and $0 \leq j \leq N$.  Note that this leads to a completion scheme for computing~$D_y^k[I]$ proposed in \cite{Mourrain:inverse-systems}, assuming $y$ is in the vanishing set of $I$:
\begin{algorithmic}
\STATE $D_y^0[I] \leftarrow \Span_\bC(\p^0)$
\FOR{$i = 1 \to k$}
\STATE $D_y^i[I] \leftarrow \{ q\in D_y \mid x_j\cdot q \in D_y^{i-1}[I] \mbox{ for all } j=1,\ldots,N \mbox{ and } q(f_i) = 0 \mbox{ for all } i=1,\ldots,n \}$
\ENDFOR
\end{algorithmic}
Moreover, the above algorithm makes apparent that if $D_y^i[I] = D_y^{i+1}[I]$ for some $i \geq 0$ then $D_y^i[I]$ is equal to all higher truncations, and so is equal to $D_y[I]$.  This gives an effective stopping criterion for computing $D_y[I]$ when it is finite dimensional.

\begin{proposition}\label{Prop:intersection}
For ideals $J_1,J_2 \subset R_0$,
\begin{itemize}
 \item $D_0[J_1 + J_2] = D_0[J_1] \cap D_0[J_2].$
 \item $D_0[J_1 \cap J_2] = D_0[J_1] + D_0[J_2].$
\end{itemize}
\end{proposition}
\begin{proof}
The first statement follows from the definition of the dual space, as does $D_0[J_1] + D_0[J_2] \subset D_0[J_1 \cap J_2].$

Let $L = D_0[J_1] + D_0[J_2]$.  It's clear that $L$ is an $R_0$-submodule, so it is the dual space of an ideal $J_L$.  The fact that $D_0[J_1] \subset L$ implies $J_L \subset J_1$ and similarly $J_L \subset J_2$.  Therefore $D_0[J_1 \cap J_2] \subset L$.
\end{proof}

\begin{corollary}\label{Coro:truncated-for-homogeneous} If $J_1$ and $J_2$ are homogeneous ideals of $R_0$, then the equality holds for the truncated dual spaces:
$$D_0^d[J_1\cap J_2] = D_0^d[J_1]+D_0^d[J_2],\ \text{for all }d\in \bN.$$
\end{corollary}
\begin{proof}
  This follows from the fact that if $q\in D_0[J]$ for a homogeneous $J$, then $q_d$, the part of $q$ of order $d$, is also in $D_0[J]$.
\end{proof}

\begin{remark}
  For truncated dual space, in general, only one inclusion holds:
  $$D_0^k[J_1\cap J_2] \supset D_0^k[J_1]+D_0^k[J_2].$$
  However, because $D_0^k[J_1\cap J_2]$ is finite dimensional, it follows that
  $$D_0^k[J_1\cap J_2] \subset D_0^l[J_1]+D_0^l[J_2]$$
  for $l$ large enough.
\end{remark}

\begin{example}
Let $I_1 = \ideal{x_1}$ and $I_2 = \ideal{x_1-x_2^2}$ in $R = \bC [x_1,x_2]$.  Then
\[ D_0^1[I_1] + D_0^1[I_2] = \Span\{1,\partial_2\}, \]
\[ D_0^1[I_1\cap I_2] = \Span\{1,\partial_1,\partial_2\}, \]
\[ D_0^2[I_1] + D_0^2[I_2] = \Span\{1,\partial_2,\partial_2^2,\partial_2^2+\partial_1\}. \]
There are strict inclusions
\[ D_0^1[I_1] + D_0^1[I_2] \subsetneq D_0^1[I_1\cap I_2] \subsetneq D_0^2[I_1] + D_0^2[I_2]. \]
\end{example}

\subsection{Dual space of quotient ideals}\label{Subsec:quotient-dual}
Recall that for $g \in R_0$, the map $\der_{g}:D_0 \to D_0$ denotes the action of $g$ on $D_0$ by pre-multiplication, or equivalently by ``differentiation'' with respect to $g$.

\begin{proposition}
 For all non-zero $g\in R_0$, the map $\der_g:D_0 \to D_0$ is surjective and $\ker \der_g = D_0[\ideal{g}]$.
\end{proposition}
\begin{proof}
 Note $g\cdot D_0$ is closed under differentiation.  If $\der_g$ is not surjective, then $g\cdot D_0$ is the dual space of some non-trivial ideal $I \subset R_0$ by Proposition~\ref{Prop:closed}.  Choose some non-zero $f \in I$.  Since $gf \neq 0$, there exists some functional $q$ with $q(gf) \neq 0$.  Then $g\cdot q(f) = q(gf) \neq 0$, which is a contradiction since $g\cdot q$ should annihilate $f$.  To show $\ker \der_g = D_0[\ideal{g}]$, if $q \in D_0[\ideal{g}]$ then $g\cdot q(f) = q(gf) = 0$ for all $f \in R_0$.  The only functional that is zero on all elements of $R_0$ is the zero functional so $g\cdot q = 0$.  Conversely if $q \notin D_0[\ideal{g}]$ then $g\cdot q(f) = q(gf) \neq 0$ for some $f \in R$, so $g\cdot q \neq 0$.
\end{proof}

\begin{theorem}\label{thm:dualOfColonIdeal}
	$D_0[I:\ideal{g}] = g\cdot D_0[I]$.
\end{theorem}

\begin{proof}
For $I$ homogeneous, the statement is shown in~\cite[Theorem 22]{Hauenstein:preprint2011}. Here we consider the general case.

If $p \in D_0[I]$, then $g\cdot p(f) = p(gf) = 0$ for all $f$ such that $gf \in I$.  These are precisely the polynomials $f$ in $I:\ideal{g}$, and so $g\cdot p \in D_0[I:\ideal{g}]$.

For any $q \in D_0[I:\ideal{g}]$, because $\der_g$ is surjective we can choose some $p \in D_0$ such that $g\cdot p = q$.  Then for all $f \in I:\ideal{g}$, we have $q(f) = p(gf) = 0$, so \[p \in D_0[g(I:\ideal{g})] = D_0[I \cap \ideal{g}] = D_0[I] + D_0[\ideal{g}]. \]
Therefore $p = p' + u$ for some $p' \in D_0[I]$ and $u \in D_0[\ideal{g}]$.  Then $q = g\cdot p = g\cdot p' + g\cdot u$ but $g\cdot u = 0$ so $q \in g\cdot D_0[I]$.
\end{proof}

\section{Local Hilbert function and its regularity}\label{Sec:Hilbert}

The Hilbert function of an ideal $I \subset R_0$ provides combinatorial information about $I$ that can be computed numerically using truncated dual spaces.  

\subsection{Primal and dual monomial order}
Let $\geq$ be a local monomial order ($1$ is the largest monomial), which we shall refer to as a {\em primal order}. For
$g = \sum_{\alpha} a_\alpha x^\alpha$, a nonzero polynomial,
the {\em initial term} with respect to $\geq$ is
the largest monomial with respect to $\geq$ that has a nonzero coefficient, namely
$$\initial_{\geq}(g) = \max_\geq\{x^\alpha~|~a_\alpha\neq 0\}.$$
For an ideal $I$, the {\it initial terms} of $I$ with respect to $\geq$
is the set of initial terms with respect to $\geq$ of all the elements of $I$, namely
$$\initial_{\geq}(I) = \{\initial_{\geq}(f)~|~f \in I\}.$$
A monomial is called a {\it standard monomial} of $I$
with respect to $\geq$ if it is not a member of $\initial_\geq(I)$.

We shall order the monomial differential functionals via the {\em dual order}:
$$
\partial^\alpha \dgeq \partial^\beta\ \Leftrightarrow\ x^\alpha \leq x^\beta,
$$
the order opposite to $\geq$.

The {\em initial term} $\initial_\dgeq(q)$ of $q$
is the largest monomial differential functional that has a nonzero coefficient.
The {\em initial support} of a dual space with respect to $\dgeq$
is the set of initial terms with respect to $\dgeq$ of all the elements in the dual space (which can be considered as a subset of $\POStoN$).

A dual basis that has distinct initial terms is called a {\em reduced dual basis}.
Using a (possibly infinite dimensional) Gaussian elimination procedure,
it is easy to see that any dual basis can be transformed into a reduced dual basis.

\begin{theorem}[Theorem~3.1 of \cite{LVZ}]\label{Thm:LVZ}
Let $I_0$ be a 0-dimensional ideal of $R_0$.
The initial support of the dual space $D_0[I_0]$ is the set of standard monomials for $I=I_0\cap R$, i.e.,
\begin{equation}\label{Eq:InitStdMon}
\initial_\succeq(D_0[I_0]) = \initial_\succeq(D_0[I]) = \{\partial^\alpha~|~x^\alpha \notin \initial_\geq(I)\}.
\end{equation}
\end{theorem}
\begin{proof}
 Note that $D_0[I]$ is finite dimensional.  Choose a monic reduced basis $B$ for $D_0[I]$ such that the lead term of each element does not occur in any other element (using Gaussian elimination).
 
 Suppose $\p^\alpha\in \LTg(D_0[I])$ so some $p \in B$ has $\LTg(p) = \p^\alpha$.  For any monic polynomial $f$ with $\LT(f) = x^\alpha$, $f$ and $p$ have no terms with the same exponent except their respective lead terms, so $p(f) = 1$ and $f \notin I$.
 
 Suppose $\p^\alpha \notin \LTg(D_0[I])$.  Let $\{p_1,\ldots,p_s\} \subset B$ be the basis elements with $\p^\alpha$ in their monomial support.  For each $p_i$ let $c_i$ be the coefficient of $\p^\alpha$ and let $\p^{\beta_i} = \LTg(p_i)$.  The following polynomial
  \[ f = x^\alpha + \sum_{i = 1}^s \frac{x^{\beta_i}}{c_i} \]
 has $p(f) = 0$ for all $p \in B$, and $\LT(f) = x^\alpha$.  By Proposition~\ref{Prop:IdealFromDual}, $f \in I_0$ so $x^\alpha \in \LT(I_0) = \LT(I)$.
\end{proof}

\begin{corollary}\label{Coro:dim-local-dual}
For an ideal $I\subset R_0$, $\dim_\bC \left(R_0/(I + \fm^{k+1})\right) = \dim_\bC D_0^{k}[I]$ where $\fm = \ideal{x_1,\ldots,x_N}$.
\end{corollary}
\begin{proof}
Choosing a graded primal order $\geq$, a vector space basis for the quotient $R_0/(I + \fm^{k+1})$ is the set of monomials
\[ \{x^\alpha~|~x^\alpha \notin \initial_\geq(I), \text{ and } |\alpha| \leq k\}. \]
By Theorem~\ref{Thm:LVZ} this is corresponds to a basis for $\initial_\succeq(D_0^{k}[I])$ which has the same dimension as $D_0^{k}[I]$.
\end{proof}

We can extend Theorem~\ref{Thm:LVZ} to ideals of arbitrary dimensions.

\begin{theorem}\label{thm:complementary-staircases}
For an ideal $I \subset R$ the monomial lattice $\bN^N$ is a disjoint union of $\LTg D_0[I]$ and $\LT I$.
\end{theorem}
\begin{proof}
By Theorem~\ref{Thm:LVZ}, $\bN^N \setminus \LTg D_0[I + \fm^{k+1}] = \LT (I+\fm^{k+1})$, so then
 \[ \bN^N \setminus \bigcup_k  \LTg D_0[I + \fm^{k+1}] = \bigcap_k \LT (I+\fm^{k+1}). \]
  By definition $\bigcup_k  \LTg D_0[I + \fm^{k+1}] = \LTg D_0[I]$, while by Lemma~\ref{Lem:ExtensionAsIntersection} 
\[\bigcap_k \LT (I+\fm^{k+1}) = \LT (IR_0 \cap R) = \LT(I)\,.\]
\end{proof}

\subsection{Hilbert function and regularity index}
\begin{definition}
For an ideal $I\subset R_0$ define the {\em Hilbert function} as
\begin{align*}
H_I(k) &= \dim_\bC (\gr(R_0/I)_k) = \dim_\bC\left(\frac{I+\fm^{k}}{I+\fm^{k+1}}\right) \\
&= \dim_\bC\left(R_0/(I+\fm^{k+1})\right) - \dim_\bC\left(R_0/(I+\fm^{k})\right).
\end{align*}
\end{definition}
This is the same as $HS_{R_0/I, \fm}$, the Hilbert-Samuel function of the $R_0$-module $R_0/I$ where $R_0$ is filtered by $\{\fm^k\}$.

The Hilbert function is determined by the initial ideal with respect to the primal monomial order (that respects the degree).
\begin{proposition} For an ideal $I\subset R_0$
  $$H_I(k) = H_{I,0}(k)= H_{\LT(I\cap R)}(k),\ \text{for all }k\in\bN.$$
\end{proposition}

Alternatively, truncated dual spaces determine the Hilbert function.  By Corollary~\ref{Coro:dim-local-dual} it can be seen that
$$H_I(k) = \dim_\bC D_0^k[I] - \dim_\bC D_0^{k-1}[I],\ \text{for }k\geq 0, $$
where $\dim_\bC D_0^{-1}[I]$ is taken to be 0.

For some $m\geq 0$ the Hilbert function is a polynomial in $k$ for all $k \geq m$ (see, e.g., \cite[Lemma 5.5.1]{Singular-book-02}), the {\em Hilbert polynomial} $\HP_I(k)$.  If the dimension of $I \subset R_0$ is $d$, then $\HP_I(k)$ is a polynomial of degree $d-1$.  In particular if $I$ is 0-dimensional then $\HP_I(k) = 0$ since $R_0/I$ is finite dimensional.

\begin{definition} The {\em regularity index} of the Hilbert function is
$$\rho_0(I) = \min\{\,m \,: \,H_I(k)=\HP_I(k)\text{ for all } k \geq m\,\}.$$
\end{definition}

The regularity index of an ideal is used as a stopping criterion for many algorithms which work iteratively by degree.  In particular we will make use of it in Algorithm~\ref{alg:saturation-test-for-a-curve}.

\section{Quotient ideals and eliminating dual spaces}\label{Sec:eliminating}
Section~\ref{Subsec:quotient-dual} described the relationship between the dual space of an ideal $D_0[I]$, and the dual space of the quotient ideal by a principal ideal $D_0[I:\ideal{g}]$.  For applications (such as in Section~\ref{Sec:embedded-test}) it is useful to compute information even about the simplest case, where $g = x_1$.  We would like to find bases for the truncated dual spaces $D_0^d[I:\ideal{x_1}]$ but this proves difficult.

Let $>$ be a graded primal order on the monomials of the local ring $R_0$, and $\succ$ be the dual order for the dual monomials of $D_0$. For any $p \in D_0$, we must have $\ord \LTg(x_1\cdot p) \leq \ord \LTg(p) - 1$, since differentiation reduces the degree of each monomial by 1, but may also annihilate the lead term.  Therefore taking the derivative of the dual space truncated at degree $d+1$ we have $x_1\cdot D_0^{d+1}[I] \subset D_0^{d}[I:\ideal{x_1}]$.  Equality may not hold since there may be some functionals $q \in D_0^{d}[I:\ideal{x_1}]$ with $q = x_1\cdot p$ for some $p \in D_0[I]$ with lead term having degree higher than $d+1$ and is annihilated by $x_1$.  In general, finding $D_0^{d}[I:\ideal{x_1}]$ from the truncated dual space of $I$ may require calculating $D_0^{c}[I]$ up to a high degree~$c$.

\medskip
To overcome the difficulty of computing truncated dual spaces of colon ideals, we consider other filtrations of $D_0$ corresponding to gradings on $R_0$ other than the total degree grading.  For $A \subset \{x_1,\ldots,x_n\}$ define $\ord_A \p^\alpha = \sum_{x_i\in A}\alpha_i$, the total order of all $\p_i$ with $x_i \in A$.  For general $q \in D_0$ define $\ord_A q$ to be the maximum order of the terms of $q$.

\begin{definition}
 Fixing $A \subset \{x_1,\ldots,x_N\}$, the {\em eliminating truncated dual spaces} of $I$ are
$$E^d_0[I,A] = \{ q\in D_0[I] : \ord_A q \leq d\}$$
for all $d\in \bN$.
\end{definition}
We often drop the word {\em truncated} when talking about eliminating dual spaces.

The truncated dual spaces $D_0^d[I]$ give a filtration of $D_0[I]$ corresponding to the maximal ideal $\fm$ of $R_0$
\[ D_0^d[I] = D_0[I + \fm^{d+1}]. \]
Similarly, the eliminating truncated dual spaces for $A$ correspond to the ideal $\ideal{A}$ in that
\[ E^d_0[I,A] = D_0[I + \ideal{A}^{d+1}]. \]

To see this, note that $E^d_0[I,A]$ is the intersection of $D_0[I]$ with $E^d_0[0,A] = D_0[\ideal{A}^{d+1}]$.  By Propositon \ref{Prop:intersection}, the intersection of these two dual spaces is $D_0[I + \ideal{A}^{d+1}]$.

For which ever grading of $R_0$ (and corresponding filtration of $D_0$) is chosen, it is useful to pick a local order $\geq$ (and corresponding dual order $\succeq$) that is compatible with the grading.  An order is compatible if for $x^\alpha \in (R_0)_i$ and $x^\beta \in (R_0)_j$ with $i < j$ then $x^\alpha > x^\beta$.  In the case of the total degree grading, such an order is a graded order.  For the grading given by $\ideal{A}$ a compatible local order is an {\em elimination order}, eliminating the variables in $A$.  In particular this is a block order in which the most significant block is a degree order on the variables in $A$ and the second block is an arbitrary order on the variables not in $A$.  Such an order ensures that $p \in E_0^d[0,A]$ if and only if $\LTg p \in E_0^d[0,A]$.

\begin{remark}
 Dual spaces offer analogs to many operations in elimination theory.  The dual space of $I \cap \bC[x_{m+1},\ldots,x_N]$ is equal to $D_0[I]|_{\p_1 = 0,\ldots,\p_m = 0}$.  The eliminating dual $E^0_0[I,A]$ is the dual space of $I + \ideal{A}$, the variety of which is the intersection of $\bV(I)$ with the coordinate subspace in which the variables in $A$ are zero.  For a more detailed discussion in the case if homogeneous ideals see \cite{Hauenstein:preprint2011}.
 
 Let $\geq$ be an local elimination order for $x_1,\ldots,x_m$ with dual order $\succeq$ and consider ring extension $R' := \bC(x_{m+1},\ldots,x_N)[x_1,\ldots,x_m] \supset R$.  In this extension with $\succeq'$ the corresponding dual order, the monomials in $\operatorname{in}_{\succeq'} D_0[IR']$ are the monomials of $\LTg D_0[I]$ considering only the $x_1,\ldots,x_m$ parts.  These can be computed from $\LTg E_0^d[I,\{x_1,\ldots,x_m\}]$ for sufficiently large $d$.
\end{remark}

Note that the eliminating dual space generalizes the usual truncated dual space since $E^d_0[I,\{x_1,\ldots,x_N\}] = D^d_0[I]$.  For general $A$, we have $E^d_0[I,A]\supset D^d_0[I]$.  Unlike $D_0^d[I]$, the eliminating truncated dual space can be infinite-dimensional.
\begin{proposition}
 If $I$ is an $m$-dimensional ideal that is in general position with respect to $x_1,\ldots,x_m$ then $\dim_\bC E^d_0[I,\{x_1,\ldots,x_m\}] < \infty$.
\end{proposition}
\begin{proof}
 For $I$ satisfying these hypotheses the intersection of $\Var(I)$ with the space $\Var(x_1,\ldots,x_m)$ is 0-dimensional.  By Theorem \ref{Thm:dimDual}, $I + \ideal{x_1,\ldots,x_m}^{d+1}$ has dual space of finite dimension.
\end{proof}
In particular, if $I$ is a curve, after a generic change of coordinates one can finitely compute its eliminating dual spaces for $A = \{x_1\}$.  The following proposition provides a method to compute eliminating dual spaces of the quotient ideal $I:\ideal{x_1}$ as well.

\begin{proposition}~\label{prop:E_0 of colon ideal} 
$E_0^{d}[I:\ideal{x_1}, \{x_1\}] = x_1 \cdot E_0^{d+1}[I,\{x_1\}]$
for all $d\in\bN$.
\end{proposition}
\begin{proof}
Let $\succeq$ be a dual order on $D_0$ eliminating $x_1$. 
	For any functional $p \in D_0$, either $\LTg(p)$ is divisible by $\p_1$, or $p$ has no terms divisible by $\p_1$.  In the first case, $\LTg(x_1\cdot p) = \LTg(p)/\p_1$.  In the second case $x_1\cdot p = 0$.  Therefore, in the view of Theorem~\ref{thm:dualOfColonIdeal}, any non-zero $q \in E_0^{d}[I:\ideal{x_1},\{x_1\}]$ must be the derivative of some $p \in E_0^{d+1}[I,\{x_1\}]$.
\end{proof}

This proposition is used in Algorithm~\ref{alg:saturation-test-for-a-curve}; see Example~\ref{example:cusp}.

Proposition~\ref{prop:E_0 of colon ideal} for curves does not hold in general (only a weaker Proposition~\ref{prop:E_0 of colon ideal_dim_m} does) and we are unable to use the eliminating dual spaces outside the specialized Algorithm~\ref{alg:saturation-test-for-a-curve}. 

\begin{proposition}~\label{prop:E_0 of colon ideal_dim_m}
\begin{align}
  \label{eq: E^d supset... E^{d+1}} E_0^{d}[I:\ideal{x_1,\ldots,x_m}, \{x_1,\ldots,x_m\}] & \supset
  \sum_{i=1}^m x_i\cdot E_0^{d+1}[I,\{x_1,\ldots,x_m\}]
\end{align}
for all $d\in\bN$.
\end{proposition}
\begin{proof} The inclusion (\ref{eq: E^d supset... E^{d+1}}) holds, since $I:\ideal{x_1,\ldots,x_m} = \bigcap_{i=1}^m I:\ideal{x_i}$ and, by Theorem~\ref{thm:dualOfColonIdeal},
\begin{align*}
  D_0[I:\ideal{x_1,\ldots,x_m}] =
  \sum_{i=1}^m D_0[I:\ideal{x_i}] =
  \sum_{i=1}^m x_i \cdot D_0[I].
\end{align*}

\end{proof}


\begin{remark} \label{rem:E_0}
Assuming it is finite, a basis for $E^d_0[I,\{x_1,\ldots,x_m\}]$ can be computed by finding a basis of the dual space of $I + \ideal{x_1,\ldots,x_m}^{d+1}$.  The dual space of a 0-dimensional ideal can be efficiently computed for example with the algorithm of \cite{Mourrain:inverse-systems} or others.
\end{remark}


\section{Detecting embedded points on curves}\label{Sec:embedded-test}

The general problem of detecting an embedded component can be formulated as follows:
\begin{quote}
Consider an ideal $I\subset R$ and a prime ideal $P_0\supset I$. Let $P_1,\ldots,P_r\supset I$ be associated ideals of $I$ such that $\sqrt{P_i}\subsetneq P_0$.

Given generators of $I$ together with general points $y_0\in \bV(P)$ and $y_i\in \bV(P_i)$ ($i=1,\cdots,r$) determine whether $P$ is an associated prime of $R/I$.
\end{quote}

The problem that we solve here is more special: we consider the case when the variety is locally a curve, namely, $\dim _{y_0} I = 1$. That means $\dim P_i=1$ for $i\neq 0$ and $\bV(P_0) = \{y_0\}$ is a point that may or may not be an embedded component. 

\medskip
Let an ideal $I$ be given by its generators $F$ and suppose, without a loss of generality, that the point in question is $y_0=0$. Let the 1-dimensional primary components in the problem be $P_1,\ldots,P_r$ with $V_i = \Var(P_i)$ containing the origin.  Saturating $I$ by the ideal $\ideal{x_1}$ eliminates all the components of $I$ that contain $\ideal{x_1}$.  After a generic linear change of coordinates, we may assume that no $V_i$ is contained in the hyperplane $x_1 = 0$ except for $V_0=\Var(Q_0)$, so $I:\ideal{x_1}^\infty \neq I$ if and only if the origin is an embedded component.

This leads to the following algorithm that employs the eliminating dual spaces.
\begin{algorithm}\label{alg:saturation-test-for-a-curve} $B = \operatorname{IsOriginEmbeddedInCurve}(I)$
\begin{algorithmic}
\REQUIRE $I$, a 1-dimensional ideal of $R$ in regular position relative to $x_1$.
\ENSURE $B = \text{``origin is an embedded component of $I$''}$, a Boolean value.
\smallskip \hrule \smallskip

\STATE $r \gets \rho_0(I)$;
\STATE $m \gets \mu_0(I)$;
\STATE $k \gets \max(r,m-1)$;
\STATE $E \gets E_0^{k}[I,\{x_1\}]$;
\RETURN $x_1\cdot E \subsetneq E_0^{k-1}[I,\{x_1\}]$

\smallskip \hrule \smallskip
\end{algorithmic}
\end{algorithm}

Here $\mu_0(I)$ denotes the {\em multiplicity} (or {\em degree}) of $I$ at the origin.  For $I$ a curve, note the (local) Hilbert polynomial of $I$ is the constant polynomial $\HP_I(k) = \mu_0(I)$.  To compute $\rho_0(I)$ and $\mu_0(I)$ we refer to the algorithm given in \cite{Krone:dual-bases-for-pos-dim} which produces the Hilbert function of $I$ from a set of generators, and in the process the Hilbert regularity index and the Hilbert polynomial of $I$.

The following two lemmas are used in the proof of correctness of Algorithm~\ref{alg:saturation-test-for-a-curve}.

\begin{lemma}\label{prop:truncated-test}
 Suppose ideals $I,J \subset R_0$ satisfy $I \subseteq J$ and $\dim_\bC J/I$ is finite.  Then $I = J$ if and only if
 \[ D^{r-1}_0[I] = D^{r-1}_0[J] \]
 where $r = \max\{\rho_0(I),\rho_0(J)\}$.
\end{lemma}
\begin{proof}
 Since $I \subseteq J$, to show $I = J$ it is enough to show that $H_I(k) = H_J(k)$ for all $k \geq 0$.  Because $\dim_\bC J/I$ is finite, $\HP_I = \HP_J$ so the Hilbert functions agree for $k \geq r$.  If additionally $D^{r-1}_0[I] = D^{r-1}_0[J]$, then the Hilbert functions also agree for $0 \leq k < r$. 
\end{proof}

\begin{lemma}\label{prop:J-reg}
 If $J$ is a one-dimensional monomial ideal that is saturated at the origin ($J = J:\fm^\infty$), then
  \[ \rho_0(J) \leq \mu_0(J)-1. \]
\end{lemma}
\begin{proof}
 We consider a {\em monomial cone decomposition} of the standard monomials of $J$.  For monomial $m \in R_0$ and a set of variables $v = \{x_{i_1},\ldots,x_{i_k}\}$ the monomial cone $C_{m,v}$ is
  \[ C_{m,v} := \{x_{i_1}^{a_1}\ldots x_{i_k}^{a_k} m \mid (a_1,\ldots,a_k) \in \bN^k\}. \]
 A monomial cone decomposition of $R_0/J$ is a finite list of pairs $$(m_1,v_1),\ldots,(m_s,v_s)$$ such that the standard monomials of $J$ are a disjoint union of the cones $$C_{m_1,v_1},\ldots,C_{m_s,v_s}.$$  The dimension of a cone $C_{m,v}$ is defined to be the size of $v$.  A cone decomposition is closely related to the Hilbert function of $J$: the maximum dimension of a cone in the decomposition is the dimension of the ideal; the number of maximal dimensional cones is the multiplicity $\mu_0(J)$; and the maximum degree of the monomials $m_1,\ldots,m_s$ bounds the regularity $\rho_0(J)$.  For $J$ a one-dimensional monomial ideal saturated at the origin, there is a cone decomposition $(m_1,\{x_{i_1}\}),\ldots,(m_s,\{x_{i_s}\})$ of $R_0/J$ consisting only of dimension 1 cones.
 
 Modify this decomposition slightly by letting $m'_j := m_j|_{x_{i_j} = 1}$, the monomial obtained from $m_j$ by removing $x_{i_j}$.  The cones $C_{m'_1,v_1},\ldots,C_{m'_s,v_s}$ also have the standard monomials of $J$ as their union, but are generally not disjoint.  To prove the proposition, it is sufficient show that for all $d \geq \mu_0(J) - 1$ each cone contains exactly one monomial of degree $d$ and these monomials are distinct, and therefore $H_J(d) = \mu_0(J)$.
 
 Let $M_k := \{m'_j \mid i_j = k \}$.  Note that $\sum_k |M_k| = \mu_0(J)$.  For each $k$, $M_k$ is closed under differentiation.  This follows from the fact that $M_k$ is the set of standard monomials of $\pi(J)$ where $\pi:\bC[x_1,\ldots,x_N] \to \bC[x_1,\ldots,\hat{x_k},\ldots,x_N]$ is the projection sending $x_k$ to 1.  If $M_k$ has a monomial of degree $d$, it also has at least one monomial of each degree $< d$.  Therefore
  \[ \max_{m \in M_k}\deg m \leq |M_k| - 1 \leq \mu_0(J) - 1. \]
 So for $d \geq \mu_0(J) - 1$, each cone contains a monomial of degree $d$.
 
 Suppose two cones in the decomposition intersect, so $n = m'_jx_{i_j}^a = m'_lx_{i_l}^b$ for some $j \neq l$ and $x_{i_j} \neq x_{i_l}$.  Then $x_{i_j}^a$ divides $m'_l$ so $a \leq \deg m'_l$.
  \[ \deg n = \deg m'_j + a \leq \deg m'_j + \deg m'_l \leq |M_{i_l}| + |M_{i_j}| - 2 \leq \mu_0 - 2. \]
 No two cones have a monomial in common of degree $d \geq \mu_0(J) - 1$.
\end{proof}

\begin{proof}[Proof of correctness of Algorithm~\ref{alg:saturation-test-for-a-curve}]
By Proposition~\ref{prop:E_0 of colon ideal} $x_1\cdot E = E_0^{k-1}[I:\ideal{x_1},\{x_1\}]$.  If this dual space is not equal to $E_0^{r-1}[I,\{x_1\}]$ then $I:\ideal{x_1} \neq I$.  This implies there is an embedded component at the origin.

Suppose instead $x_1\cdot E = E_0^{k-1}[I,\{x_1\}]$.  We will use Lemma~\ref{prop:truncated-test} to prove that $I:\ideal{x_1} = I$.  The truncated dual space of degree $r-1$ is contained in the eliminating dual space of degree $r$, so $D_0^{r-1}[I:\ideal{x_1}] = D_0^{r-1}[I]$.  We know that $I \subseteq I:\ideal{x_1}$.  Because they differ by at most a zero-dimensional component, $\dim_{\bC}(I:\ideal{x_1})/I$ is finite.

Finally it must be shown that $k \geq \max(\rho_0(I),\rho_0(I:\ideal{x_1}))$.  It is clear that $k \geq \rho_0(I)$.  To show $k \geq \rho_0(I:\ideal{x_1})$, let $J = \initial(I):\fm^{\infty}$, which has the same Hilbert polynomial as $I$ and $I:\ideal{x_1}$ and satisfies
 \[ \initial(I) \subseteq \initial(I:\ideal{x_1}) \subseteq J, \]
 \[ H_I \geq H_{I:\ideal{x_1}} \geq H_J. \]
By Lemma~\ref{prop:J-reg}, $\rho_0(J) \leq \mu_0(I) - 1$.  Since $H_{I:\ideal{x_1}}$ is sandwiched between $H_I$ and $H_J$, once they stabilize to $\mu_0(I)$, so must $H_{I:\ideal{x_1}}$.  This implies the regularity of $I:\ideal{x_1}$ is bounded by $k$.
\end{proof}

\begin{example}\label{example:cusp}
 Let $I = \ideal{x^2 - z^3,y-z^2} \subset \bC[x,y,z]$ which defines a curve in $\bC^3$ with a singular point at the origin.  The deflation algorithm from \cite{Leykin:NPD} will identify the origin as a possible embedded component.  Note that $\rho_0(I) = 1$, $\mu_0(I) = 2$ and no irreducible component of $\Var(I)$ is contained in the plane $x = 0$.  To test whether the origin is embedded, we compute the eliminating dual $E_0^1[I,\{x\}]$.  This is the set of all dual functionals with all terms having $\p_x$-degree $\leq 1$.
  \[ E_0^1[I,\{x\}] = \Span\{ 1, \p_z^2 + \p_y, \p_z, \p_x, \p_x\p_z^2 + \p_x\p_y, \p_x\p_z \}, \]
  \[ x\cdot E_0^1[I,\{x\}] = \Span\{ 1, \p_z^2 + \p_y, \p_z \}. \]
  Since $x\cdot E_0^1[I,\{x\}] = E_0^0[I,\{x\}]$ we conclude that the origin is not an embedded component of $I$.
\end{example}

\begin{example}
 For this example we compute with an implementation of Algorithm \ref{alg:saturation-test-for-a-curve} in {\em Macaulay2}.  Let $I$ be the ideal of the cyclic4 system, generated by
 \[ \{x_1 + x_2 + x_3 + x_4,\; x_1x_2 + x_2x_3 + x_3x_4 + x_4x_1, \]
 \[ x_2x_3x_4 + x_1x_3x_4 + x_1x_2x_4 + x_1x_2x_3,\; x_1x_2x_3x_4 - 1\}. \]
 $I$ is a curve with several singular points, which are discovered using the algorithm described in \cite{Leykin:NPD} up to some numerical precision.  One such point is $p =$
 \[ (-0.0000000000000000122 + 1.000000000000000222i, \]
 \[ -0.0000000000000001128 + 0.999999999999999889i, \]
 \[  0.0000000000000000459 - 0.999999999999999889i, \]
 \[ -0.0000000000000000935 - 1.000000000000000000i), \]
 approximately $(i,i,-i,-i)$.  Let $I'$ denote the ideal obtained from $I$ by a random affine change of coordinates that fixes $p$.  This ensures that $I'$ is in general position with respect of $x_1$.  Using the algorithm of \cite{Krone:dual-bases-for-pos-dim}, the regularity index is $\rho_p(I') = 2$ and the multiplicity is $\mu_p(I') = 1$, so $k = \max(2,0) = 2$.
 
 Computing $E_p^1[I',\{x_1\}]$ and $x_1\cdot E_p^2[I',\{x_1\}]$ the dimensions are 3 and 2 respectively, so they are not equal.  Therefore the point being approximated by $p$ is an embedded component of $I$.  The code for this example can be found at \cite{ECTwww}.
\end{example}

\bibliographystyle{plain}
\bibliography{bib}

\end{document}